\documentclass[11pt,a4paper]{article}

\usepackage{amsfonts}
\usepackage{amssymb,amsthm,amsmath,graphicx,bm,ebezier,mathtools}
\usepackage{hyperref}
\usepackage{color,enumerate}
\usepackage[ruled,vlined]{algorithm2e}
\usepackage{tikz}
\usepackage[all]{xy}
\usepackage{lineno} 

\newtheorem{theorem}{Theorem}

\newtheorem{proposition}[theorem]{Proposition}
\newtheorem{corollary}[theorem]{Corollary}
\newtheorem{lemma}[theorem]{Lemma}

\theoremstyle{remark}
\newtheorem{example}[theorem]{Example}
\newtheorem{remark}[theorem]{Remark}

\def\CaN{\mathcal{N}}
\def\CaF{\mathcal{F}}
\def\CaC{\mathcal{C}}
\def\CaH{\mathcal{H}}

\def\N{\mathbb{N}}

\def\Q{\mathbb{Q}}

\def\FG{\mathrm{FG}}
\def\PF{\mathrm{PF}}
\def\SG{\mathrm{SG}}

\def\CaN{\mathcal{N}}
\def\CaH{\mathcal{H}}
\def\CaC{\mathcal{C}}

\def\N{\mathbb{N}}

\def\Q{\mathbb{Q}}

\title{Some properties of affine $\mathcal C$-semigroups}
\date{}
\author{
J. I. Garc\'{\i}a-Garc\'{\i}a
\\
D. Marín-Aragón
\\
A. Sánchez-Loureiro
\\
A. Vigneron-Tenorio
}

\begin{document}
\maketitle
\begin{abstract}
Numerical semigroups have been extensively studied throughout the literature, and many of their invariants have been characterized. In this work, we generalize some of the most important results about symmetry, pseudo-symmetry, or fundamental gaps, to affine $\CaC$-semigroups. In addition, we give algorithms to compute the tree of irreducible $\CaC$-semigroups and $\CaC$-semigroups with a given Frobenius vector.
\end{abstract}



\section*{Introduction}

A $\CaC$-semigroup is a non-empty subset of $\N^p$ (for some non-zero natural number $p$), containing $0$ and closed under addition, such that $\CaC\setminus S$ is finite; $\CaC\subset \N^p$ denotes the integer cone generated by $S$. These semigroups are the natural generalization to higher dimensions of the numerical semigroups. Moreover, some objects related to numerical semigroups can be generalized to $\CaC$-semigroups. For example, the elements in $\CaC\setminus S$ are called {\em gaps} of $S$, and the cardinality of its gap set is called {\em genus} of $S$. We denote this set by $\CaH(S)$, and its cardinality, by $g(S)$.

There are other objects whose generalization needs to consider a total order on $\N^p$. For example, the Frobenius number of a numerical semigroup is the maximum integer that is not in it. Still, its generalization over the $\CaC$-semigroups is not unique if we do not fix a total order. So, fixed $\prec$ a total order on $\N^p$, the Frobenius element of $S$ is $\max_{\prec} (\CaC\setminus S)$.

Even though $\CaC$-semigroups frequently appear in semigroup theory, it is not until the publication of \cite{GenSemNp} that they have become an object of study in their own right. This paper defines {\em generalized numerical semigroups} as the $\CaC$-semigroups where the cone $\CaC$ is $\N^p$. Since 2016, several works have been devoted to study different properties of $\CaC$-semigroups in general and generalized numerical semigroups in particular. For example, in \cite{CFU}, the authors show that any $\N^p$-semigroup has a unique minimal system of generators and provide an algorithm to compute its set of gaps from a set of generators of the $\N^p$-semigroup. In \cite{Csemigroup}, an extension of Wilf's conjecture for numerical semigroups is given to $\CaC$-semigroups, and in \cite{Wilf_GNS}, another one is introduced for $\N^p$-semigroups. This paper also studies the irreducibleness of the $\N^p$-semigroups. More recent papers about $\N^p$-semigroups are \cite{corner}, \cite{algoritmos_GNS}, \cite{irreducible_GNS}, \cite{Almost_GNS}, and \cite{SL}. For any $\CaC$-semigroup, in \cite{check_C_semigr}, the authors mainly provide an algorithm to check if an affine semigroup given by a generating set is a $\CaC$-semigroup and to compute its gap set.

The main goal of this work is to generalize several results of numerical semigroups to $\CaC$-semigroups. A $\CaC$-semigroup is $\CaC$-reducible (simplifying reducible) when it can be expressed as an intersection of two $\CaC$-semigroups containing it properly (see \cite{resolucion_maxima}); $S$ is $\CaC$-irreducible (simplifying irreducible) in another case. In this work, we also characterize irreducible $\CaC$-semigroups from their genus and from their generalized Frobenius numbers. 
We also study when a subset of a cone $\CaC$ is the gap set of a $\CaC$-semigroup or determines it.
These results are complemented by some algorithms for checking the corresponding properties.

Moreover, some algorithms for computing some objects related to $\CaC$-semigroups are provided. In particular, it is defined a tree whose vertex set is the set of all irreducible $\CaC$-semigroups with a fixed Frobenius vector. An algorithm to compute this tree is also introduced. For any integer cone $\CaC$ and any non-null element $\mathbf{f}\in \CaC$, we give a procedure to obtain all $\CaC$-semigroups with Frobenius element equal to $\mathbf{f}$.

The results of this work are illustrated with several examples. For this purpose, we have implemented all the algorithms shown in this work in our library \emph{CommutativeMonoids} dedicated to the study of numerical and affine semigroups (see \cite{PROGRAMA}) developed by the authors in Python \cite{python} and C++. A notebook containing all the examples of this work can be found at the following link \url{https://github.com/D-marina/CommutativeMonoids/blob/master/CClassCSemigroups/SomePropertiesCSemigroup.ipynb}.

The content of this work is organized as follows: Section \ref{sec_prelimiraries} is devoted to provide the reader with the necessary background for the correct understanding of the work. In Section \ref{C_irreducible}, we introduce the concept of symmetric and pseudo-symmetric $\CaC$-semigroups, and some characterizations of these concepts are given. We turn our attention in Section \ref{sec:tree} to the irreducible $\CaC$-semigroups, we prove that we can build a tree with all these semigroups with a fixed Frobenius vector, and we show an algorithm for computing them. Similarly, an algorithm for computing all the $\CaC$-semigroups with a fixed Frobenius vector is given in Section \ref{sec:CsemigroupsGivenFrobenius}. Finally, in Section \ref{fundamental_gaps}, we study the fundamental gaps of a $\CaC$-semigroup, and for any set $X\subset\CaC$, we give conditions to determine if $\CaC\setminus X$ is a $\CaC$-semigroup.

\section{Preliminaries}\label{sec_prelimiraries} 

In this work, $\Q$, $\Q_\ge$, and $\N$ denote the sets of rational numbers, non-negative rational numbers, and non-negative integer numbers, respectively. For any $n\in \N$, $[n]$ denotes the set $\{1,\ldots , n\}$.

A non-degenerated rational cone in $\Q_\ge^p$ is the convex hull of finitely many half lines in $\Q_\ge^p$ emanating from the origin. These cones can also be determined from their supporting hyperplanes. We consider that the integer points of a rational cone form an integer cone in $\N^p$. It is well known that any integer cone $\CaC\subset \N^p$ is finitely generated if and only if a rational point exists in each of its extremal rays. Moreover, any subsemigroup of $\CaC$ is finitely generated if and only if there exists an element in the subsemigroup in each extremal ray of $\CaC$. Both results are proved in \cite[Chapter 2]{Bruns}, where an in-depth study on cones can also be found. We assume that any integer cone considered in this work is finitely generated.

Throughout this work, we use some particular gaps in $\CaH(S)$ whose definitions are the same for numerical semigroups \cite{libro_rosales}:
\begin{itemize}
\item $\mathbf{x}\in \CaH(S)$ is a \emph{fundamental gap} if $2\mathbf{x},3\mathbf{x}\in S$. The set of these  elements is denoted by $\FG(S)$.
\item $\mathbf{x}\in \CaH(S)$ is a \emph{pseudo-Frobenius element} if $\mathbf{x}+(S\setminus\{0\})\subset S$, the set of pseudo-Frobenius elements of $S$ is denoted by $\PF(S)$, and its cardinality is known as the type of $S$, $t(S)$.
\item $\mathbf{x}\in \CaH(S)$ is a \emph{special gap} of $S$ if $\mathbf{x}\in\PF(S)$ and $2\mathbf{x}\in S$. We denote by $\SG(S)$ the set of special gaps of $S$. 
\end{itemize}

In this work, we consider different orders on some sets. On a non-empty set $L\subset \N^p$ and $\mathbf{x},\mathbf{y}\in \N^p$, consider the partial order $\mathbf{x}\le_L \mathbf{y}$ if $\mathbf{y}- \mathbf{x}\in L$. Besides, we also fix $\preceq$ a total order on $\N^p$ determined by a monomial order. A monomial order is a total order on the set of all (monic) monomials in a given polynomial ring (see \cite{CoxLO}). From the properties of a monomial order, the (induced) total order $\preceq$ on $\N^p$ satisfies:
\begin{itemize}
\item if $ \mathbf{a} \preceq \mathbf{b}$ and $\mathbf{c}\in \N^p$, then $\mathbf{a}+\mathbf{c}\preceq \mathbf{b}+\mathbf{c}$,
\item if $ \mathbf{c}\in\N^p$, then $0\preceq \mathbf{c}$.
\end{itemize}
Every monomial order can be represented via matrices. For a nonsingular integer ($p\times p$)-matrix $M$ with rows $M_1,\ldots ,M_p,$ the $M$-ordering $\prec$ is defined by $\mathbf{a}\prec \mathbf{b}$ if and only if there exists an integer $i$ belonging to $[p-1],$ such that $M_1\mathbf{a}=M_1\mathbf{b},\ldots ,M_i\mathbf{a}=M_i\mathbf{b}$ and $M_{i+1}\mathbf{a}<M_{i+1}\mathbf{b}.$ 

From the fixed total order on $\N^p$, the Frobenius vector of $S$, $F(S)$, is the maximal element in $\CaH(S)$ respect to $\preceq$, and we set $n(S)$ as the cardinality of $\CaN(S)=\{\mathbf{x}\in S\mid \mathbf{x}\preceq F(S) \}$.

The following lemma generalizes to $\CaC$-semigroups Proposition 2.26 in \cite{libro_rosales}.

\begin{lemma}
Let $S$ be a $\CaC$-semigroup. Then, $g(S)\le t(S)n(S)$.
\end{lemma}

\begin{proof}
Just as it occurs for numerical semigroups,
for any $\mathbf{x}\in \CaH(S)$, there exist $(\mathbf{f},\mathbf{s})\in \PF(S)\times S$ such that $\mathbf{f}=\mathbf{x}+\mathbf{s}$, and $\mathbf{f}_\mathbf{x}=\min _{\preceq}\{ \mathbf{f}\in \PF(S)\mid \mathbf{f}-\mathbf{x}\in S \}$. Hence, the map $\CaH(S)\to \PF(S)\times \CaN(S)$, defined by $x\mapsto (\mathbf{f}_\mathbf{x},\mathbf{f}_\mathbf{x}-\mathbf{x})$ is injective, and thus $g(S)\le t(S)n(S)$.
\end{proof}

\section{Symmetric and pseudo-symmetric $\CaC$-semigroups}\label{C_irreducible}

Fix $S\subset \N^p$ a $\CaC$-semigroup with genus $g$. In this section, we characterize the symmetric and pseudo-symmetric $\CaC$-semigroups using their genus. We say that $S$ is $\CaC$-irreducible when $\PF(S)$ is equal to $\{F(S)\}$ or $\{F(S),F(S)/2\}$ (see \cite{resolucion_maxima}). If $\PF(S)=\{F(S)\}$, we say that $S$ is symmetric, and pseudo-symmetric when $\PF(S)=\{F(S),F(S)/2\}$.

For any element $\mathbf{n}$ in $\CaC$, let $I_S(\mathbf{n})$ be the set $\{ \mathbf{s}\in S\mid \mathbf{s}\le _\CaC \mathbf{n}\}$.
\begin{remark}\label{iff}
Note that, for any $\mathbf{s}\in S$, $\mathbf{s}\in I_S(F(S))$ if and only if $F(S)-\mathbf{s}\in \CaH(S)$. Thus, $g\ge \sharp I_S(F(S))$.
\end{remark}

We have the following characterizations of symmetric and pseudo-symmetric $\CaC$-semigroups.

\begin{proposition}\label{symmetric}
Let $S$ be a $\CaC$-semigroup with genus $g$. Then, $S$ is symmetric if and only if $g=\sharp I_S(F(S))$.
\end{proposition}

\begin{proof}
Assume that $S$ is symmetric. Thus, $F(S)$ is the unique pseudo-Frobenius element of $S$. Furthermore, for any $\mathbf{x}\in \CaH(S)$, there exists $\mathbf{s}\in S$ such that $\mathbf{x}+\mathbf{s}=F(S)$, that is $\mathbf{s}\in I_S(F(S))$, and then $\sharp I_S(F(S))\ge g$. Since $g\ge \sharp I_S(F(S))$, we conclude that $g=\sharp I_S(F(S))$.

Conversely, note that $I_S(F(S))=\{\mathbf{s}\in S\mid F(S)-\mathbf{s}\in \CaH(S)\}$, and suppose that $g= \sharp I_S(F(S))$. Hence, every $\mathbf{x}\in \CaH(S)\setminus \{F(S)\}$ satisfies  $F(S)-\mathbf{x}\in S$, and then $\mathbf{x}$ is not a pseudo-Frobenius element of $S$.
\end{proof}

\begin{proposition}\label{pseudo-symmetric}
Let $S$ be a $\CaC$-semigroup with genus $g$. Then, $S$ is pseudo-symmetric if and only if $g=1+\sharp I_S(F(S))$ and $F(S)/2\in \N^p$.
\end{proposition}

\begin{proof}
Assume that $S$ is pseudo-symmetric, thus $\PF(S)=\{F(S),F(S)/2\}$, and $g>\sharp I_S(F(S))$. For all $\mathbf{x}\in \CaH(S)\setminus\{F(S)/2\}$, there exists some $\mathbf{s}\in S$ such that $\mathbf{x}+\mathbf{s}= F(S)$, or $\mathbf{x}+\mathbf{s}= F(S)/2$. If the first option is satisfied, $\mathbf{s}\in I_S(F(S))$. In other case, $\mathbf{x}+\mathbf{s}+F(S)/2= F(S)$ and then $\mathbf{s}+F(S)/2$ also belongs to $I_S(F(S))$. Besides, $F(S)/2+\mathbf{s}\neq F(S)$ for every $\mathbf{s}\in S$. Hence, $\sharp I_S(F(S))\ge g-1$.

Conversely, suppose that $g= \sharp I_S(F(S))+1$ and $F(S)/2\in \N^p$. Hence, there exists only one $\mathbf{x}\in \CaH(S)\setminus \{F(S)\}$ with $\mathbf{x}+\mathbf{s}\neq F(S)$ for all $\mathbf{s}\in S$. Hence, $\PF(S)=\{F(S),\mathbf{x}\}$. If $\mathbf{x}\neq F(S)/2$, then there is $\mathbf{s}\in S$ such that $F(S)/2+\mathbf{s}=F(S)$, and $F(S)/2\in S$, but it is not possible. So, $\mathbf{x}= F(S)/2$.
\end{proof}

Consider the Apéry set of a $\CaC$-semigroup $S$ relative to $\mathbf{b} \in S \setminus \{0\}$ as $\mathrm{Ap}(S,\mathbf{b}) = \{\mathbf{a} \in S \mid \mathbf{a}-\mathbf{b} \in \CaH(S)\}$. The following proposition shows the relationship between the pseudo-Frobenius elements of $S$ and its Apéry set.

\begin{proposition}\cite[Proposition 16]{resolucion_maxima}\label{Prop ApPF}
Let $S$ be a $\CaC$-semigroup and $\mathbf{b} \in S \setminus \{0\}$. Then,
\begin{equation*}\label{ecu1} \PF(S) = \{\mathbf{a} - \mathbf{b} \mid \mathbf{a} \in \mathrm{maximals}_{\le_S} \mathrm{Ap}(S,\mathbf{b}) \}.\end{equation*}
\end{proposition}

From this result, we can generalize the corollaries 4.12 and 4.19 in \cite{libro_rosales}.

\begin{corollary}
Let $S$ be a $\CaC$-semigroup and $\mathbf{b} \in S \setminus \{0\}$. The semigroup $S$ is symmetric if and only if $\mathrm{maximals}_{\le_S} \mathrm{Ap}(S,\mathbf{b})=\{F(S)+\mathbf{b}\}$.
\end{corollary}

\begin{corollary}
Let $S$ be a $\CaC$-semigroup and $\mathbf{b} \in S \setminus \{0\}$. The semigroup $S$ is pseudo-symmetric if and only if $\mathrm{maximals}_{\le_S} \mathrm{Ap}(S,\mathbf{b})=\{F(S)+\mathbf{b},F(S)/2+\mathbf{b}\}$.
\end{corollary}

The Frobenius number of a numerical semigroup is the maximum non-negative integer that is not an element of the semigroup. We define the (generalized) Frobenius number of a $\CaC$-semigroup $S$ as $\CaF(S)=\sharp I_S(F(S))+g(S)$. We can easily rewrite the previous propositions \ref{symmetric} and \ref{pseudo-symmetric} from this definition.

\begin{corollary}
Let $S$ be a $\CaC$-semigroup with genus $g$. Then, $S$ is symmetric if and only if $2g=\CaF(S)$.
\end{corollary}

\begin{corollary}
Let $S$ be a $\CaC$-semigroup with genus $g$. Then, $S$ is pseudo-symmetric if and only if $2g=1+\CaF(S)$ and $F(S)/2\in \N^p$.
\end{corollary}

These corollaries specialized to numerical semigroups or $\N^p$-semigroups are equivalent to Corollary 4.5 in \cite{libro_rosales}, and the theorems 5.6 and 5.7 in \cite{irreducible_GNS}, respectively.

We illustrate the previous results with one easy example.

\begin{example}
Let $\CaC\subset \N^2$ be the cone with extremal rays $\overrightarrow{(7,3)}$ and $\overrightarrow{(15,1)}$.

The $\CaC$-semigroup $S_1$ minimally generated by
$$
\begin{multlined}
\Lambda_{S_1}=\{
(3, 1), (4, 1), (5, 1), (6, 1), (7, 1), (7, 3), (8, 1), (8, 3), (9, 1), (10, 1),\\ (11, 1), (12, 1), (12, 5), (13, 1), (14, 1), (15, 
   1)
 \}
\end{multlined}
$$ is symmetric, while the $S_2$ minimally generated by
$$
\begin{multlined}
\Lambda_{S_2}=\{
(3, 1), (5, 2), (6, 1), (7, 1), (7, 2), (7, 3), (8, 1), (9, 1), (10,1) \\
 (11, 1), (12, 1), (13, 1), (14, 1), (15, 1)
 \}
\end{multlined}
$$ is pseudo-symmetric.

Note that $\PF(S_1)=\SG(S_1)=\CaH(S_1)=\{(5,2)\}$, but $\CaH(S_2)=\{(4, 1), (5, 1), (8, 2)\}$, $\PF(S_2)=\{(4, 1), (8, 2)\}$, and $\SG(S_2)=\{(8, 2)\}$.

\end{example}

\section{Trees of irreducible $\CaC$-semigroups}\label{sec:tree}

This section describes a tree whose vertex set is the set of all irreducible $\CaC$-semigroups with a fixed Frobenius vector.

Again, consider $\CaC\subset \N^p$ an integer cone and $\mathbf{f}\in \CaC\setminus \{0\}$. Consider a monomial order $\preceq$ on $\N^p$ and decompose the set $I_\CaC(\mathbf{f})$ as $I_\CaC(\mathbf{f})=I_1(\mathbf{f})\sqcup I_2(\mathbf{f})$ with $I_1(\mathbf{f})=\{\mathbf{x}\in I_\CaC(\mathbf{f})\mid \mathbf{0}\neq \mathbf{x}\preceq \mathbf{f}/2\}$ and $I_2(\mathbf{f})=\{\mathbf{x}\in I_\CaC(\mathbf{f})\mid \mathbf{x}\succ \mathbf{f}/2\}$ (when $\mathbf{f}/2\notin \N^p$, consider $\preceq$ as the monomial order extended to $\Q_{\ge}^p$). We define the $\CaC$-semigroup $S(\mathbf{f})$ as $\big( \CaC \setminus\{\mathbf{f}\}\big)\setminus I_1(\mathbf{f})$. This semigroup will be the root of our tree of irreducible $\CaC$-semigroups; this root depends on the fixed monomial order, as the following example shows.

\begin{example}\label{e:S(f)}
Let $\CaC\subset \N^2$ be the integer cone with extremal rays $\overrightarrow{(1,0)}$ and $\overrightarrow{(1,2)}$, and $\mathbf{f}=(4,2)$. Then, $\mathbf{f}/2=(2,1)$ and
$$I_\CaC(\mathbf{f})=\{(1, 0), (1, 1), (1, 2), (2, 0), (2, 1), (2, 2), (3, 0), (3, 1), (3, 2), (4, 2)\}.$$
Let $\prec_1$ and $\prec_2$ be the orders defined by the matrices $\left(\begin{array}{cc}
       1  &  1\\
       1  & 0
    \end{array}
    \right)$ and $\left(\begin{array}{cc}
       1  &  1\\
       0  & 1
    \end{array}
    \right)$, respectively.
In the first case, $I_1(\mathbf{f})_{\prec_1}=\{(1, 0), (1, 1), (1, 2), (2, 0), (2, 1)\}$ and 
\begin{multline*}
  S(\mathbf{f})_{\prec_1}=\langle (3, 0), (4, 0), (5, 0), (3, 1), (4, 1),\\ (5, 1), (2, 2), (3, 2), (2, 3), (3, 3), (4, 3), (2, 4), (3, 4), (3, 5), (3, 6) \rangle.
\end{multline*}
In the other one, $I_1(\mathbf{f})_{\prec_2}=\{(1, 0), (1, 1), (2, 0), (2, 1), (3, 0)\}$ and
\begin{multline*}
S(\mathbf{f})_{\prec_2}=\langle (4, 0), (5, 0), (6, 0), (7, 0), (3, 1),\\ (4, 1), (5, 1), (6, 1), (1, 2), (2, 2), (3, 2), (2, 3), (3, 3) \rangle.
\end{multline*}
\end{example}

The set $S(\mathbf{f})$ satisfies interesting properties collected in the following lemma.

\begin{lemma}
The $\CaC$-semigroup $S(\mathbf{f})$ is irreducible. Moreover, $\mathbf{f}$ is the Frobenius vector of $S(\mathbf{f})$ for any monomial order, and $S(\mathbf{f})$ is the unique irreducible $\CaC$-semigroup satisfying all its gaps belong to $I_1(\mathbf{f})\cup \{\mathbf{f}\}$.
\end{lemma}

\begin{proof}
By definition of $S(\mathbf{f})$, $\mathbf{f}$ is the unique maximum in $\CaH(S(\mathbf{f}))$ respect to $\le _\CaC$. So, it is also the unique maximum in $\CaH(S(\mathbf{f}))$ respect to $\le _{\N^p}$. This fact implies that $\mathbf{f}$ is the Frobenius vector of $S(\mathbf{f})$ for any monomial order.

Note that the set of gaps of $S(\mathbf{f})$ is the set $\CaH(S(\mathbf{f}))=I_1(\mathbf{f})\cup \{\mathbf{f}\}$, and $I_{S(\mathbf{f})}(\mathbf{f})=I_2(\mathbf{f})\setminus\{\mathbf{f}\}$. Besides, for any $\mathbf{x}\in \CaH(S(\mathbf{f}))$, the element $\mathbf{f}-\mathbf{x}$ belongs to $I_{S(\mathbf{f})}(\mathbf{f})$. In other case, $\mathbf{f} = \mathbf{f} - \mathbf{x} + \mathbf{x} \prec \mathbf{f}/2+\mathbf{f}/2=\mathbf{f}$. Furthermore, since $\mathbf{x}\in I_{S(\mathbf{f})}(\mathbf{f})$ if and only if $\mathbf{f}-\mathbf{x}\in \CaH(S(\mathbf{f}))$, we have that the cardinality of $\CaH(S(\mathbf{f}))$ is equal to $1+\sharp I_{S(\mathbf{f})}(\mathbf{f})$ when $\mathbf{f}\in 2\N^p$, or equal to $\sharp I_{S(\mathbf{f})}(\mathbf{f})$ in the other case. By Proposition \ref{pseudo-symmetric} or Proposition \ref{symmetric} (respectively), $S(\mathbf{f})$ is an irreducible $\CaC$-semigroup.

The uniqueness of $S(\mathbf{f})$ is given by its definition.
\end{proof}

The following proposition gives us an irreducible $\CaC$-semigroups from an existing one, such that both have the same Frobenius vector.

\begin{proposition}\label{son}
Let $S$ be a $\CaC$-semigroup irreducible with Frobenius vector $\mathbf{f}$, and $\mathbf{x}\in I_S(\mathbf{f})$ be one of its minimal generators such that:
\begin{enumerate}
    \item $2\mathbf{x}-\mathbf{f}\notin S$.
    \item $3\mathbf{x}\neq 2\mathbf{f}$.
    \item $4\mathbf{x}\neq 3\mathbf{f}$.
\end{enumerate}
Then, $S'= (S\setminus\{\mathbf{x}\})\cup \{\mathbf{f}-\mathbf{x}\}$ is a $\CaC$-semigroup irreducible with Frobenius vector $\mathbf{f}$.
\end{proposition}

\begin{proof}
Note $F(S')= \mathbf{f}$. We prove that $S'$ is closed under addition. Since $\mathbf{x}=(\mathbf{f}-\mathbf{x})+ (2\mathbf{x}-\mathbf{f})$, $2\mathbf{x}-\mathbf{f}$ can not belong to $S$, that is, the second condition is necessary.

Trivially, given two elements in $S\setminus\{\mathbf{x}\}$, their addition belongs to the same set. Besides, $\mathbf{f}-\mathbf{x}+\mathbf{s}\in S\setminus\{\mathbf{x}\}$ for any $\mathbf{s}\in S\setminus\{\mathbf{x}\}$. In other case, $\mathbf{f}-\mathbf{x}+\mathbf{s}= \mathbf{x}$ or $\mathbf{f}-\mathbf{x}+\mathbf{s}\in \CaH (S)$, for some $\mathbf{s}\in S\setminus\{\mathbf{x}\}$. If $\mathbf{f}-\mathbf{x}+\mathbf{s}= \mathbf{x}$, then $\mathbf{s}=2\mathbf{x}-\mathbf{f}\notin S$. 
If $\mathbf{f}-\mathbf{x}+\mathbf{s}\in \CaH (S)$, then there exists $\mathbf{s}'\in S$ such that $\mathbf{f}-\mathbf{x}+\mathbf{s}+\mathbf{s}'\in \PF(S)$. When $\mathbf{f}-\mathbf{x}+\mathbf{s}+\mathbf{s}'=\mathbf{f}/2$, we have $2(\mathbf{s}+\mathbf{s}')=2\mathbf{x}-\mathbf{f}\notin S$, and when $\mathbf{f}-\mathbf{x}+\mathbf{s}+\mathbf{s}'=\mathbf{f}$, $\mathbf{x}=\mathbf{s}+\mathbf{s}'$. Both conclusions are not possible.

To finish this proof, we show that $2(\mathbf{f}-\mathbf{x})\in S\setminus\{\mathbf{x}\}$. Assume that $2(\mathbf{f}-\mathbf{x})\notin S\setminus\{\mathbf{x}\}$, so $2(\mathbf{f}-\mathbf{x})= \mathbf{x}$, or $2(\mathbf{f}-\mathbf{x})+\mathbf{s}\in \PF(S)$ for some $\mathbf{s}\in S$. Since $3\mathbf{x}\neq 2\mathbf{f}$, $2(\mathbf{f}-\mathbf{x})\neq \mathbf{x}$. The semigroup $S$ to be irreducible implies that $2(\mathbf{f}-\mathbf{x})+\mathbf{s}\in \{\mathbf{f},\mathbf{f}/2\}$, but $2(\mathbf{f}-\mathbf{x})+\mathbf{s}$ is not equal to $\mathbf{f}$ because of $2\mathbf{x}-\mathbf{f}\notin S$. Hence, $2(\mathbf{f}-\mathbf{x})+\mathbf{s}=\mathbf{f}/2$. Since $4\mathbf{x}\neq 3\mathbf{f}$, $\mathbf{s}\neq 0$, and from $2\mathbf{x}-\mathbf{f}\notin S$, we obtain $2(\mathbf{f}-\mathbf{x})+\mathbf{s}\neq\mathbf{f}/2$. We conclude $2(\mathbf{f}-\mathbf{x})\in S\setminus\{\mathbf{x}\}$.

Since $\sharp I_S(\mathbf{f})=\sharp I_{S'}(\mathbf{f})$ and $\sharp \CaH (S)= \sharp \CaH (S')$, $S'$ is irreducible by the propositions \ref{symmetric} and \ref{pseudo-symmetric}. 
\end{proof}

From this point forward, we will use the notation $m(S)$ to represent the minimum element (with respect to the partial order $\preceq$) in the minimal generating set of $S$. This element is often referred to as the multiplicity of $S$.

We denote by $\mathfrak{I}(\mathbf{f})$ the set of the irreducible $\CaC$-semigroups with Frobenius vector $\mathbf{f}$. Given $S\in \mathfrak{I}(\mathbf{f})$, consider $S_0=S,$ and $S_n=(S_{n-1}\setminus\{m(S_{n-1})\})\cup \{\mathbf{f}-m(S_{n-1})\}$ when $m(S_{n-1})\in I_1(\mathbf{f})$, or $S_n=S_{n-1}$ in other case, for $n>1$. Note that $S_n=S_{n-1}$ if $S_n=S(\mathbf{f})$. Since $I_1(\mathbf{f})$ is a finite set, the set $\{S_0,S_1,\ldots\}$ is also finite. Let $G=(V,E)$ be the digraph given by the set of vertices $V=\mathfrak{I}(\mathbf{f})$, and edge set $E=\big\{(A,B)\in V\times V\mid m(A)\prec \mathbf{f}/2 \text{ and } B=(A\setminus\{m(A)\})\cup \{\mathbf{f}-m(A)\}\big\}$.

\begin{theorem}
Let $\preceq$ be a monomial order on $\N^p$, $\CaC\subset \N^p$ be an integer cone, and $\mathbf{f}\in \CaC$ be a non zero element. The digraph $G$ is a rooted tree with root $S(\mathbf{f)}$.
\end{theorem}

\begin{proof}
Let $S$ be an element belonging to $\mathfrak{I}(\mathbf{f})$. If $m(S)\notin I_1(\mathbf{f})$, then $S=S(\mathbf{f})$. Assume that $m(S)\in I_1(\mathbf{f})$. In that case, $m(S)\prec \mathbf{f}/2$, that is, $2m(S)-\mathbf{f}\notin S$, $3m(S)\neq 2\mathbf{f}$, and $4m(S)\neq 3\mathbf{f}$. By Proposition \ref{son}, $S_1=(S\setminus\{m(S)\})\cup \{\mathbf{f}-m(S)\}$ is irreducible. That means $(S,S_1)\in E$. Following this construction, $G$ is a tree whose root is $S(\mathbf{f)}$.
\end{proof}

We obtain an algorithm from previous construction and results to compute a tree of all irreducible $\CaC$-semigroups with a given Frobenius vector and a fixed monomial order (Algorithm \ref{a:treeIrreducible}).

\begin{algorithm}[H]\label{C_semigruop_fixed_Frob}
	\KwIn{A monomial order $\preceq$ on $\N^p$, an integer cone $\CaC$ and $\textbf{f}\in \CaC$.}
	\KwOut{A tree of irreducible $\CaC$-semigroups with Frobenius vector $\textbf{f}$.}

\Begin{
    $X \leftarrow \{S(\textbf{f})\}$\;
    $Y \leftarrow \emptyset$\;
    \While{$X\neq\emptyset$}
    {
        $S \leftarrow \mbox{First}(X)$\;
        $A \leftarrow \{x\in S\mid x\in I_2(\textbf{f})\cap\Lambda_S,2x-\textbf{f}\notin S, 3x\neq\textbf{f}, 4x\neq 3\textbf{f}, \textbf{f}-x\prec m(S)\}$\;
        \If{$A=\emptyset$}
        {
            $Y \leftarrow Y\cup\{S\}$\;
        }
        \Else
        {
            \For{$x\in A$}
            {
                $H \leftarrow (\mathcal{H}(S)\setminus\{\textbf{f}-x\})\cup\{ x \}$\;
                $S' \leftarrow \CaC$-semigroup with $\mathcal{H}(S')=H$\;
                $X\leftarrow X\cup\{S'\}$\;
            }
        }
        $X \leftarrow X\setminus\{S\}$\;
    }

    \Return{$Y$}
}
\caption{Computing a tree of irreducible $\CaC$-semigroups with a given Frobenius vector.}\label{a:treeIrreducible}
\end{algorithm}

The following example shows how to apply Algorithm \ref{a:treeIrreducible} using the semigroups of Example \ref{e:S(f)}.

\begin{example}

Let $S(\textbf{f})_{\prec_1}$ be the semigroup spanned by
$$\{ (3, 0), (4, 0), (5, 0), (3, 1), (4, 1),(5, 1), (2, 2), (3, 2), (2, 3), (3, 3), (4, 3), (2, 4),$$ $$(3, 4), (3, 5), (3, 6)\}.$$
Applying Algorithm \ref{a:treeIrreducible}, we have that  $I_{2_{\prec_1}}(\textbf{f})=\{(2, 2), (3, 0), (3, 1), (3, 2), (4, 2)\}$.
Hence, $S(\textbf{f})_{\prec_1}$ has three children:
\begin{itemize}
    \item $\langle (4, 0), (5, 0), (6, 0),(7, 0), (3, 1), (4, 1), (5, 1),$ $ 
    (6, 1), (1, 2), (2, 2), (3, 2), (2, 3), (3, 3)\rangle$,
    \item $\langle (3, 0), (4, 0),$ $ (5, 0), (1, 1), (3, 2), (2, 3), (2, 4), (3, 6) \rangle$,
    \item $\langle (2, 0), (3, 0), (3, 1), (4, 1), (3, 2),$ $(2, 3),$ $ (3, 3), (2, 4), (3, 4), (3, 5), (4, 5), (3, 6)\rangle$.
\end{itemize}
After repeating this procedure, the tree in Figure \ref{f:treelex} is obtained. Since the definition of $S(\textbf{f}$) depends on the monomial order, we get a new tree if we change it. For example, when we use the order $\prec_2$, Figure \ref{f:treeprec2} appears.
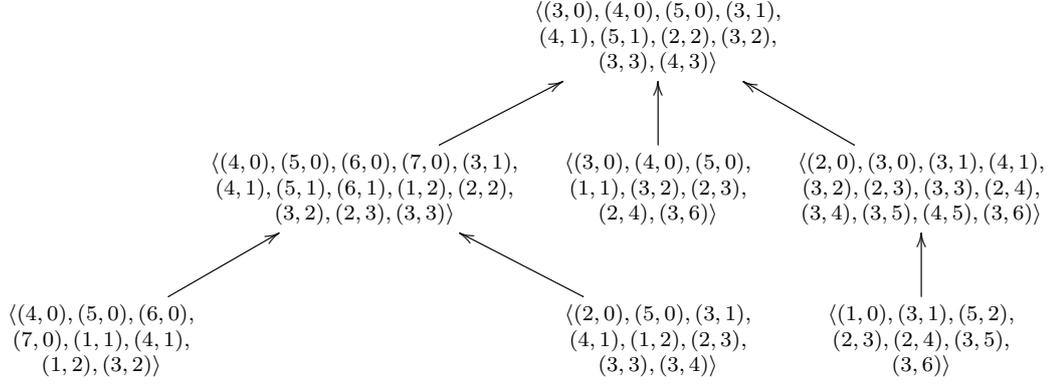
\begin{figure}
    \centering
    \scriptsize
$\xymatrix@C=0pc{  
 & &  \txt{$\langle (3, 0), (4, 0), (5, 0), (3, 1)$, \\ $(4, 1),(5, 1), (2, 2), (3, 2)$,\\ $  (3, 3), (4, 3)\rangle$} &  \\
& \txt{$\langle (4, 0), (5, 0), (6, 0), (7, 0), (3, 1), $ \\ $ (4, 1), (5, 1), (6, 1), (1, 2), (2, 2),$ \\ $ (3, 2), (2, 3), (3, 3) \rangle$}\ar[ru] &  \txt{$\langle (3, 0), (4, 0), (5, 0), $ \\ $ (1, 1), (3, 2), (2, 3), $ \\ $ (2, 4), (3, 6) \rangle$}\ar[u] &  \txt{$\langle (2, 0), (3, 0), (3, 1), (4, 1), $ \\ $ (3, 2), (2, 3), (3, 3), (2, 4), $ \\ $ (3, 4), (3, 5), (4, 5), (3, 6) \rangle$}\ar[lu]\\
\txt{$\langle (4, 0), (5, 0), (6, 0), $ \\ $ (7, 0), (1, 1), (4, 1), $ \\ $ (1, 2), (3, 2) \rangle$}\ar[ru] & & \txt{$\langle (2, 0), (5, 0), (3, 1), $ \\ $ (4, 1), (1, 2), (2, 3), $ \\ $ (3, 3), (3, 4) \rangle$}\ar[lu] &  \txt{$\langle (1, 0), (3, 1), (5, 2), $ \\ $ (2, 3), (2, 4), (3, 5), $ \\ $ (3, 6) \rangle$}\ar[u] \\
}$
    \caption{Tree of irreducible $\CaC$-semigroups with $\prec_1$.}
    \label{f:treelex}
\end{figure}
\begin{figure}
    \centering
    \scriptsize
$\xymatrix@C=0pc{  
& \txt{$\langle (4, 0), (5, 0), (6, 0), (7, 0), (3, 1),$ \\ $ (4, 1), (5, 1), (6, 1), (1, 2), (2, 2), $ \\ $ (3, 2), (2, 3), (3, 3) \rangle$} &\\
 \txt{$\langle(4, 0), (5, 0), (6, 0),$ \\ $ (7, 0), (1, 1), (4, 1),$ \\ $ (1, 2), (3, 2)\rangle$}\ar[ru] & \txt{$\langle(3, 0), (4, 0), (5, 0), (3, 1), (4, 1),$ \\ $ (5, 1), (2, 2), (3, 2), (2, 3), (3, 3),$ \\ $ (4, 3), (2, 4), (3, 4), (3, 5), (3, 6)\rangle$}\ar[u] & \txt{$\langle(2, 0), (5, 0), (3, 1),$ \\ $ (4, 1), (1, 2), (2, 3),$ \\ $ (3, 3), (3, 4)\rangle$} \ar[lu]\\
  \txt{$\langle(3, 0), (4, 0), (5, 0),$ \\ $ (1, 1), (3, 2), (2, 3),$ \\ $ (2, 4), (3, 6)\rangle$}\ar[ru] & & \txt{$\langle(2, 0), (3, 0), (3, 1), (4, 1),$ \\ $ (3, 2), (2, 3), (3, 3), (2, 4),$ \\ $ (3, 4), (3, 5), (4, 5), (3, 6)\rangle$}\ar[lu]\\
  & & \txt{$\langle(1, 0), (3, 1), (5, 2),$ \\ $ (2, 3), (2, 4), (3, 5),$ \\ $ (3, 6)\rangle$}\ar[u]
 }$
    \caption{Tree of irreducible $\CaC$-semigroups with $\prec_2$.}
    \label{f:treeprec2}
\end{figure}
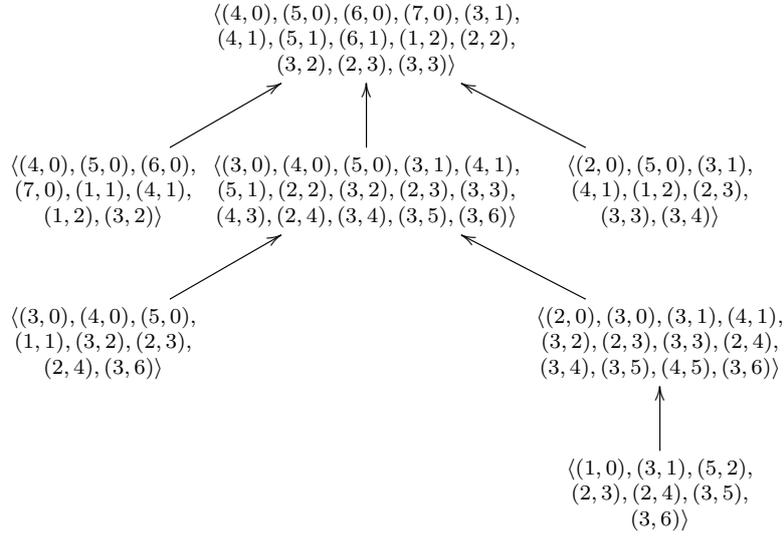
\end{example}

\section{Fundamental gaps of $\CaC$-semigroups}\label{fundamental_gaps}

In this section, we generalize to $\CaC$-semigroups several results related to the fundamental gaps of a numerical semigroup (see \cite[Chapter 4]{libro_rosales}). The first results allow us to check when $\CaC\setminus X$ is a $\CaC$-semigroup for any finite subset $X\subset \CaC$.  Denote by $D(X)$ the set $\{\mathbf{a}\in \CaC \mid n\mathbf{a}\in X \text{ for some }n\in \N \}$.

\begin{proposition}\label{prop:characterizationSminusX}
Let $\CaC\subset \N^p$ be an integer cone and $X$ be a finite subset of $\CaC\setminus\{0\}$. Then, $\CaC\setminus X$ is a $\CaC$-semigroup if and only if $\mathbf{x}-\mathbf{s}\in X$ for every $(\mathbf{x},\mathbf{s})\in (X,\CaC\setminus X)$ with $\mathbf{s}\le _\CaC \mathbf{x}$.
\end{proposition}

\begin{proof}
Let $S$ be the set $\CaC\setminus X$, and assume that $S$ is a $\CaC$-semigroup. Set $(\mathbf{x},\mathbf{s})\in (X,S)$ with $\mathbf{s}\le _\CaC \mathbf{x}$. Since $\mathbf{s}\le _\CaC \mathbf{x}$, we have that $\mathbf{x}-\mathbf{s}\in \CaC$. If $\mathbf{x}-\mathbf{s}\notin X$, then $\mathbf{x}=\mathbf{s}+\mathbf{s}'$ for some $\mathbf{s}'\in S$, and $S$ is not a semigroup. So, $\mathbf{x}-\mathbf{s}\in X$ for any $(\mathbf{x},\mathbf{s})\in (X,\CaC\setminus X)$ with $\mathbf{s}\le_\CaC \mathbf{x}$.

Conversely, since $\mathbf{x}-\mathbf{s}$ belongs to $X$ for every $(\mathbf{x},\mathbf{s})\in (X,S)$ with $\mathbf{s}\le _\CaC \mathbf{x}$, $S$ is an additive submonoid of $\N^p$ with finite complement in $\CaC$, that is, $S$ is a $\CaC$-semigroup.
\end{proof}

From above proposition, $\CaC\setminus X$ to be a $\CaC$-semigroup implies that $X=D(X)$; for example, if we consider $\CaC$ the cone generated by $\{(1,0), (1,1), (1,2)\}$ and $X=\{(2,0), (2,1)\}$, $\CaC\setminus X$ is not a semigroup because of $D(X)=\{(2,0), (2,1),(1,0)\}$. We now provide an algorithm to determine if $\CaC\setminus X$ is a $\CaC$-semigroup (Algorithm \ref{semig_from_gaps1}).
\begin{algorithm}[h]\label{semig_from_gaps1}
	\KwIn{$\CaC\subset \N^p$ an integer cone, and $X$ a finite subset of $\CaC\setminus\{0\}$.}
	\KwOut{True if $\CaC\setminus X$ is a $\CaC$-semigroup, and False in other case.}

\Begin{
    \If{$X\neq D(X)$} {\Return False}
    \While{$X\neq \emptyset$}
        {$\mathbf{x}\leftarrow \text{First}(X)$\;
        $A\leftarrow\{\mathbf{s}\in \CaC\setminus X\mid \mathbf{s}\le _\CaC \mathbf{x}\}$\;
        $s\leftarrow \text{First}(A)$\;
        \While{$A\neq \emptyset$}
            {
            \If{$\mathbf{x}-\mathbf{s}\notin X$}
                {\Return False}
            {$s\leftarrow \text{First}(A\setminus\{\mathbf{s}\})$}\;
            }
        {$X\leftarrow X\setminus\{\mathbf{x}\}$}\;
        }
    \Return True.   
}
\caption{Checking if $\CaC\setminus X$ is a $\CaC$-semigroup.}
\end{algorithm}

Since, for each $\mathbf{x}\in X$, the set $\{\mathbf{s}\in \CaC\setminus X\mid \mathbf{s}\le _\CaC \mathbf{x}\}$ can be very very large, the condition $\mathbf{x}-\mathbf{s}\notin X$ has to be checked many, many times in Algorithm \ref{semig_from_gaps1}, and many iterations are required for the worst cases. To improve the computational resolution of this problem, we provide an alternative algorithm (Algorithm \ref{semig_from_gaps2}) obtained from the following lemma and \cite[Lemma 3]{Csemigroup}.

\begin{lemma}\label{cadena_ascendente}
Fix a total order $\preceq$ on $\N^p$, and let $X=\{\mathbf{x}_1\preceq \mathbf{x}_2\preceq \cdots \preceq \mathbf{x}_t\}$ be a subset of an integer cone $S_0=\CaC\subset \N^p$. Assume that $S_t=\CaC\setminus X$ is a $\CaC$-semigroup. Then, $S_i=S_{i-1}\setminus \{\mathbf{x}_i\}$ is a $\CaC$-semigroup, and $\mathbf{x}_i$ is a minimal generator of $S_{i-1}$, for every $i\in [t]$.
\end{lemma}

\begin{proof}
Note that $\mathbf{x}_i$ is the Frobenius vector of $S_{i}$ respect to $\preceq$. Hence, $S_{i-1}=S_i\cup \{\mathbf{x}_i\}$ is a $\CaC$-semigroup and $\mathbf{x}_i$ is a minimal generator of $S_{i-1}$, for every $i\in [t]$.
\end{proof}
\begin{algorithm}[h]\label{semig_from_gaps2}
	\KwIn{A total order $\preceq$ on $\N^p$, $\Lambda_\CaC$ the minimal generating set of the integer cone $\CaC\subset \N^p$, and $X=\{\mathbf{x}_1\preceq \cdots \preceq \mathbf{x}_t\}\subset \CaC\setminus \{0\}$.}
	\KwOut{If $\CaC\setminus X$ is a $\CaC$-semigroup, its minimal generating set, and the empty set in another case.}

\Begin{
    \If{$X\subset \Lambda_\CaC$} {\Return the minimal generating set of $\CaC \setminus X$}
    \If{$X\neq D(X)$}
    {\Return $\{\}$}
    $\Lambda \leftarrow \Lambda_\CaC$\;
    \For{$1\le i\le t$}
        {
        \If{$\mathbf{x_i}\notin \Lambda$}{\Return $\{\}$}
        $\Lambda\leftarrow \text{the minimal generating set of } \langle \Lambda\rangle \setminus \{\mathbf{x}_i\}$\;
        $X\leftarrow X\setminus \{\mathbf{x}_i\}$\;
        \If{$X\subset \Lambda$} {\Return the minimal generating set of $\langle \Lambda\rangle \setminus X$}
        }
}
\caption{Checking if $\CaC\setminus X$ is a $\CaC$-semigroup.}
\end{algorithm}
We illustrate this algorithm with the following example.

\begin{example}
Let $\CaC$ be the cone generated by $\Lambda_{\CaC}=\{(1,0), (1,1), (1,2)\}$ and $X=\{(1,0),(1,1),$ $(1,2),(2,0),(2,1),(2,2),(2,3),(2,4)\}$. 
Since $X\not\subset\Lambda_{\CaC}$ and $X=D(X)$, if we apply Algorithm \ref{semig_from_gaps2}, we obtain that:
\begin{itemize}
    \item $t=0$, $\Lambda = \{(2, 0), (3, 0), (1, 1), (2, 1), (1, 2)\}$,
    \item $t=1$, $\Lambda = \{(2, 0), (3, 0), (2, 1), (3, 1), (1, 2), (2, 2), (2, 3)\}$,
    \item $t=2$, $\Lambda = \{(2, 0), (3, 0), (2, 1), (3, 1), (2, 2), (3, 2), (2, 3), (3, 3), (2, 4),$ $(3, 4), (3, 5), (3, 6)\}$.
\end{itemize}
Therefore,
\begin{multline*}
    \CaC\setminus X = \big\langle (3, 0), (4, 0), (5, 0), (3, 1), (4, 1), (5, 1), (3, 2), (4, 2), (5, 2),(3, 3),\\ (4, 3), (5, 3), (3, 4), (4, 4), (5, 4), (3, 5), (4, 5), (5, 5), (3, 6), (4, 6),\\ (5, 6), (4, 7),(5, 7), (4, 8), (5, 8), (5, 9),  (5, 10) \big\rangle.
\end{multline*}
\end{example}

Fix $S\subset \N^p$ a $\CaC$-semigroup minimally generated by $\Lambda=\{\mathbf{s}_1,\ldots, \mathbf{s}_q ,\mathbf{s}_{q+1},\ldots,\mathbf{s}_t\}$, and consider $\Lambda_\CaC=\{\mathbf{a}_1,\ldots, \mathbf{a}_q ,\mathbf{a}_{q+1},\ldots,\mathbf{a}_m\}$ the minimal generating set of $\CaC$,  with $\mathbf{s}_i,\mathbf{a}_i\in \tau_i$ for $i=1,\ldots , q$ (we assume that the integer cone $\CaC$ has $q$ extremal rays $\{\tau_1,\ldots ,\tau_q\}$).

Note that, the elements $\mathbf{x}$ of $\SG(S)$ are those elements in $\CaH (S)$ such that $S\cup \{\mathbf{x}\}$ is again a $\CaC$-semigroup. These gaps play an important role in decomposing a $\CaC$-semigroups into irreducible $\CaC$-semigroups (\cite{check_C_semigr}).

Similarly to numerical semigroups, given two $\CaC$-semigroups $S$ and $T$ with $S\subsetneq T$, any $\mathbf{x}\in \max_{\le_\CaC} (T\setminus S)$ belongs to $\SG(S)$, that is to say $S\cup \{\mathbf{x}\}$ is a $\CaC$-semigroup. Note that if $\mathbf{x}\in \max_{\le_\CaC} (T\setminus S)$, then $2\mathbf{x}\in S$. From this fact, we can prove the following proposition.

\begin{proposition}
Let $S$ be a $\CaC$-semigroup and $G$ be a subset of $\CaH(S)$. Then, $S\in \max_{\subseteq} \{ T\text{ is a $\CaC$-semigroup}\mid G\subseteq \CaH(T)\}$ if and only if $\SG(S)\subseteq G$.
\end{proposition}

\begin{proof}
We know that $\mathbf{x}\in \SG(S)$ if and only if $S\cup \{\mathbf{x}\}$ is a $\CaC$-semigroup. So, if $S\in \max_{\subseteq} \{ T\text{ is a $\CaC$-semigroup}\mid G\subseteq \CaH(T)\}$, then $\mathbf{x}\in G$. In other case, $S\subsetneq S\cup \{\mathbf{x}\}$ and $S$ is not maximal.

Assume that $S$ is not maximal but $\SG(S)\subseteq G$, so there exists $T$ a $\CaC$-semigroup such that $S\subsetneq T$ and $G\subseteq \CaH(T)$. Let $\mathbf{x}\in \max_{\le_\CaC} (T\setminus S)$, thus $\mathbf{x}\in \SG(S)\cap T$, but it is not possible ($\SG(S)\subseteq G\subseteq T$).
\end{proof}

There is another interesting subset related to the set of gaps of $S$. A subset $X$ of $\CaH(S)$ is said to determine $\CaH(S)$ if $S= \max_{\subseteq} \{ T\text{ is a $\CaC$-semigroup}\mid X\subseteq \CaH(T)\}$. These subsets were introduced in \cite{fundamental_gaps_numericos} for numerical semigroups.

\begin{proposition}\label{X_determine}
Let $X$ be a finite subset of an integer cone $\CaC\subset \N^p$. Then, $X$ determines the set of gaps of a $\CaC$-semigroup if and only if $\CaC \setminus D(X)$ is a $\CaC$-semigroup.
\end{proposition}

\begin{proof}
Fix $\Lambda_\CaC=\{\mathbf{a}_1,\ldots, \mathbf{a}_q ,\mathbf{a}_{q+1},\ldots,\mathbf{a}_m\}$ the minimal generating set of $\CaC\subset \N^p$. 

Assume that $X$ determines $\CaH(S)$ for a $\CaC$-semigroup $S$, so $X\subset D(X)\subset \CaH(S)$ and $S\subset \CaC\setminus D(X)$. Let $S'$ be the non-empty set $$\{0\}\cup \bigcup_{i=1}^q\Big\{h_i\mathbf{a}_i+\CaC\mid h_i=\min_{n\in \N}\{(n\mathbf{a}_i+\CaC)\cap X=\emptyset\}\Big\}.$$
Note that $S'$ is a $\CaC$-semigroup. Let $\mathbf{a}$ and $\mathbf{b}$ be two elements in $S'$, so $\mathbf{a}= h_i \mathbf{a}_i + \sum _{k=1}^m \alpha_k \mathbf{a}_k$, and $\mathbf{b}= h_j \mathbf{a}_j + \sum _{k=1}^m \beta_k \mathbf{a}_k$ for some $h_i,h_j,j,i,\alpha_k,\beta_k\in \N$ with $i,j\in [q]$ and $k\in [m]$. Hence, $\mathbf{a} + \mathbf{b} = h_i \mathbf{a}_i + (h_j \mathbf{a}_j + \sum _{k=1}^m (\alpha_k +\beta_k) \mathbf{a}_k)\in h\mathbf{a}_i+\CaC$. Furthermore, $\CaC\setminus S'$ is finite. Get any $\mathbf{a} \in \Lambda_\CaC$, then $\mathbf{a} = \sum _{i=1}^q \alpha_i \mathbf{a}_i$ for some $\alpha_1,\ldots ,\alpha_q\in \Q_\ge$, and hence $k \mathbf{a} = \sum _{i=1}^q \beta_i \mathbf{a}_i$ for some $\beta_1,\ldots ,\beta_q,k\in \N$. We can assume that $\beta_i\ge h_i$. In that case, $\CaC\setminus S'$ is a subset of the finite set $\{\sum _{i=1}^q \gamma _i \mathbf{a}_i\mid 0\le \gamma_i\le \beta_i \}$. We obtain that $S'$ is a finitely generated $\CaC$-semigroup; let $\Lambda _{S'}$ its minimal generating set. The set $X$ is a subset of $\CaH(S')$ by construction.

For every $\mathbf{a}\in \CaC\setminus D(X)$ and we can define $S_\mathbf{a}$ as the semigroup generated by $\{\mathbf{a}\}\cup \Lambda _{S'}$. Since $X\subset \CaH(S_\mathbf{a})$, and $X$ determines $\CaH(S)$, we have that $S_\mathbf{a}\subset S$. Hence, $\CaC\setminus D(X) \subset S$ and then $\CaC \setminus D(X)$ is a $\CaC$-semigroup.

Conversely, any $\CaC$-semigroup $T$ such that $X\subset \CaH(T)$ satisfies that $D(X)\subset \CaH(T)$. Thus, $X$ determines the set of gaps of the $\CaC$-semigroup $\CaC \setminus D(X)$.
\end{proof}

The sets determining the set of gaps of a $\CaC$-semigroup are related to its set of fundamental gaps. 

\begin{lemma}
Let $S$ be a $\CaC$-semigroup and $X$ be a subset of $\CaH(S)$. Then, $X$ determines $\CaH(S)$ if and only if $\FG(S)\subseteq X$. 
\end{lemma}

\begin{proof}
By Proposition \ref{X_determine}, if $X$ determines $\CaH(S)$, then $\CaH(S)=D(X)$. Thus, for all $\mathbf{x}\in \CaH(S)$, $h\mathbf{x}\in X$ for some $h\in \N$. In particular, for every fundamental gap of $S$, the integer $h$ has to be one. Hence, $\mathbf{x}\in X$.

Conversely, since $X\subset \CaH(S)$, we know that $D(X)\subseteq \CaH(S)$. Let $\mathbf{x}\in \CaH(S)$ and consider $h=\max\{k\in \N\mid k\mathbf{x}\in \CaH(S)\}$. In that case, $h\mathbf{x}\in \CaH(S)$, and $2h\mathbf{x},3h\mathbf{x} \in S$. Therefore,  $h\mathbf{x}\in \FG(S)\subseteq X$, $\mathbf{x}\in D(X)$, and $\CaH(S)\subseteq D(X)$.
\end{proof}

Analogously to the case of numerical semigroups,
it happens that $\FG(S)$ is the smallest subset of $\CaH(S)$ determining $\CaH(S)$. Also, the relationship between the special and fundamental gaps of a $\CaC$-semigroup is equivalent to their relationship for numerical semigroups.

\begin{lemma}\label{l:SGandFG}
Let $S$ be a $\CaC$-semigroup. Then, $\SG(S)=\max_{\le _S} \FG(S)$.
\end{lemma}

\begin{proof}
Trivially, for any $\mathbf{x}\in \SG(S)$, $2\mathbf{x},3\mathbf{x}\in S$, and then $\SG(S)\subseteq \FG(S)$. Assume that for a $\mathbf{x}\in \SG(S)$, there exists some $\mathbf{y}\in \FG(S)$ with $\mathbf{x}\le _{S} \mathbf{y}$. So, $\mathbf{x}+\mathbf{s}= \mathbf{y}$ for some $\mathbf{s}\in S$. Since $\mathbf{x}$ is a pseudo-Frobenius element of $S$, $\mathbf{y}\in S$. It is not possible, then $\mathbf{x}\in\max_{\le _S} \FG(S)$.
\end{proof}

A $\CaC$-irreducible semigroup can also be characterized using its fundamental gaps using the above lemma.

\begin{corollary}
$S$ is a $\CaC$-irreducible semigroup if and only if the cardinality of $\max_{\le _S} \FG(S)$ is equal to one.
\end{corollary}

The next example illustrates many results appearing in this section.
\begin{example}
    Let $\CaC$ be the cone with extremal rays $\tau_1=\langle (1,0) \rangle$ and $\tau_2=\langle (1,1) \rangle$ and $X=\{ (1, 1), (3, 0), (3, 1), (3, 2), (5, 1), (5, 2) \}$. Since $D(X)=\{(1,0), (1, 1), (3, 0), (3, 1), (3, 2), (5, 1), (5, 2) \}$, we have that
    \begin{multline*}
    \{(x,s)\in(D(X),\CaC\setminus D(X))\mid s\leq_{\CaC} D(X)\}=\\ 
    \{((0,0),(1, 1)), ((0,0),(3, 0)), ((0,0),(3, 1)), ((0,0),(3, 2)),\\
    ((0,0),(5, 1)), ((0,0),(5, 2)),
    ((2,0),(3,0)),
    ((2,0),(3,1)), ((2,0),(5,1)),\\ ((2,0),(5,2)),
     ((2,1),(3,2)), ((2,1),(5,1)), ((2,1),(5,2)),\\
    ((2,2),(3,2)), ((2,2), (5,2)),
    ((4,0),(5,1))\}
    \end{multline*}
   Therefore, by Proposition \ref{prop:characterizationSminusX}, $\CaC\setminus D(X)$ is a $\CaC$-semigroup and, by Proposition \ref{X_determine}, $X$ determines the set of gaps of a $\CaC$-semigroup. If we call this semigroup $S$, we have that $\mathcal{H}(S)=D(X)$. It is not difficult to check that $S=\langle (2, 0), (5, 0), (2, 1), (2, 2), (3, 3) \rangle$ and that, in this case, $\FG(S)=X$. Moreover, we can compute the set of pseudo-Frobenius elements of $S$, and we get $\PF(S)=\{(5,1), (5,2)\}$, so $\SG(S)=\{(5,1), (5,2)\}$. On the other hand, $\FG(S)=\{(1, 1), (3, 0), (3, 1), (3, 2), (5, 1), (5, 2)\}$ and $\max_{\leq_S}\FG(S) = \{(5,1), (5,2)\}$, as we knew by Lemma \ref{l:SGandFG}.
\end{example}

\section{Computing all the $\CaC$-semigroups with a given Frobenius vector}\label{sec:CsemigroupsGivenFrobenius}

Let $\CaC\subset \N^p$ be an integer cone, $\preceq$ be a monomial order on $\N^p$, and $S$ be a $\CaC$-semigroup with Frobenius vector $F(S)\in C\setminus \{0\}$. Note that $F(S)$ is a minimal generator of $S\cup \{F(S)\}$. 
 
Conversely to Lemma \ref{cadena_ascendente}, we can consider the following sequence of $\CaC$-semigroups for some $t\in \N$: $S_t=S$, $S_{i-1}= S_i\cup {F(S_i)}$ for all $i=1,\dots ,t$, and $S_0=\CaC$.
Such a sequence can be constructed for any $\CaC$-semigroup with Frobenius vector $F(S)$. 
So, from a minimal system of generators of $\CaC$, we obtain new $\CaC$-semigroups just by removing a minimal generator $\mathbf{s}$ fufilling that $\mathbf{s}\preceq F$. Performing this process as many times as possible, we obtain all the $\CaC$-semigroups with Frobenius vector $F$. Note that this process is finite due to the finitiness of the set $\{ \mathbf{s} \in \CaC \mid \mathbf{s} \preceq F\}$. This idea allows us to provide an algorithm for computing all the $\CaC$-semigroups with a fixed Frobenius vector (Algorithm \ref{C_semigruop_fixed_Frob}). Moreover, this algorithm can be modified to obtain all the $\CaC$-semigroups with the Frobenius vector less than or equal to a fixed Frobenius vector. For any set of ordered pairs, $A$, $\pi_1(A)$ denotes the set of the first projection of its elements.
\begin{algorithm}[h]\label{C_semigruop_fixed_Frob}
	\KwIn{A total order $\preceq$ on $\N^p$, $\Lambda_\CaC$ the minimal generating set of the integer cone $\CaC\subset \N^p$, and $F\in \CaC\setminus\{0\}$.}
	\KwOut{The set of $\CaC$-semigroups with Frobenius vector equal to $F$.}

\Begin{
    $\Lambda_{\preceq F}\leftarrow\{\mathbf{s}\in \Lambda_\CaC\mid \mathbf{s}\preceq F\}$\;
    $\mathcal T\leftarrow \{(\Lambda_\CaC,\Lambda_{\preceq F})\}$\;
    $\mathcal{G} \leftarrow \emptyset$\;
        \While{$\mathcal T\neq \emptyset$}
        {
        $(\Lambda,\Lambda')\leftarrow \textrm{First} (\mathcal{T})$\;
        $D \leftarrow \Lambda'$\;
        \While{$D\neq \emptyset$}
            {
            $\mathbf{s}\leftarrow \textrm{First} (D)$\;
                $\Lambda'' \leftarrow \{\text{minimal generating set of } \langle \Lambda\rangle \setminus \{\mathbf{s}\}\}$\;
            \If{$\Lambda''\notin \pi_1 (\mathcal T)$}
                {
                \If{$\mathbf{s}=F$}
                    {
                    $\mathcal{G}\leftarrow\mathcal{G}\cup \Lambda''$\;
                    }
                    $\Lambda''_{\preceq F}\leftarrow\{\mathbf{s}\in \Lambda''\mid \mathbf{s}\preceq F\}$\;
                    \If{$\Lambda''_F\neq \emptyset$}
                        {$\mathcal{T} \leftarrow \mathcal{T}\cup \{(\Lambda'',\Lambda''_{\preceq F})\}$\;}
                }        
            $D \leftarrow D \setminus \{\mathbf{s}\}$\;
        }
        $\mathcal{T} \leftarrow \mathcal{T}\setminus \{(\Lambda,\Lambda')\}$\;
        }
    \Return $\mathcal{G}$
}
\caption{Computing $\CaC$-semigroups with a given Frobenius vector.}
\end{algorithm}

\begin{example}
Let $\CaC$ be the cone generated by $\{(1,0),(1,1),(1,2)\}$ and $F=(2,1)$. Then, applying Algorithm \ref{C_semigruop_fixed_Frob}, we get that the set of all $\CaC$-semigroups with Frobenius vector $(2,1)$ is $\{
\{(2, 0), (3, 0), (1, 1), (1, 2)\},
 \{(1, 0), (3, 1), (1, 2),$ $(2, 3)\},
 \{(3, 0), (4, 0),$ $(5, 0), (1, 1), (3, 1), (1, 2), (3, 2)\},
 \{(2, 0), (3, 0), (3, 1),$ $(4, 1),$ $(1, 2), (2, 2),$ $(2, 3),$ $(3, 3)\},$ $
 \{(3, 0),
  (4, 0),
  (5, 0),$ $
  (3, 1),
  (4, 1),
  (5, 1),
  (1, 2),
  (2, 2),$ $
  (3, 2),$ $
  (2, 3),
  (3, 3)\},$ $
 \{(3, 0),
  (4, 0),
  (5, 0),
  (3, 1),
  (4, 1),
  (5, 1),
  (2, 2),
  (3, 2),
  (4, 2),$ $
  (2, 3),
  (3, 3),$ $
  (4, 3),$ $
  (2, 4),$ $
  (3, 4),
  (3, 5),
  (3, 6)\}\}$. These semigroups are shown in Table \ref{tab:LastExample}.
\begin{table}[h]
    \centering
    \begin{tabular}{|c|c|c|}
    \hline
    \includegraphics[width=0.3\textwidth]{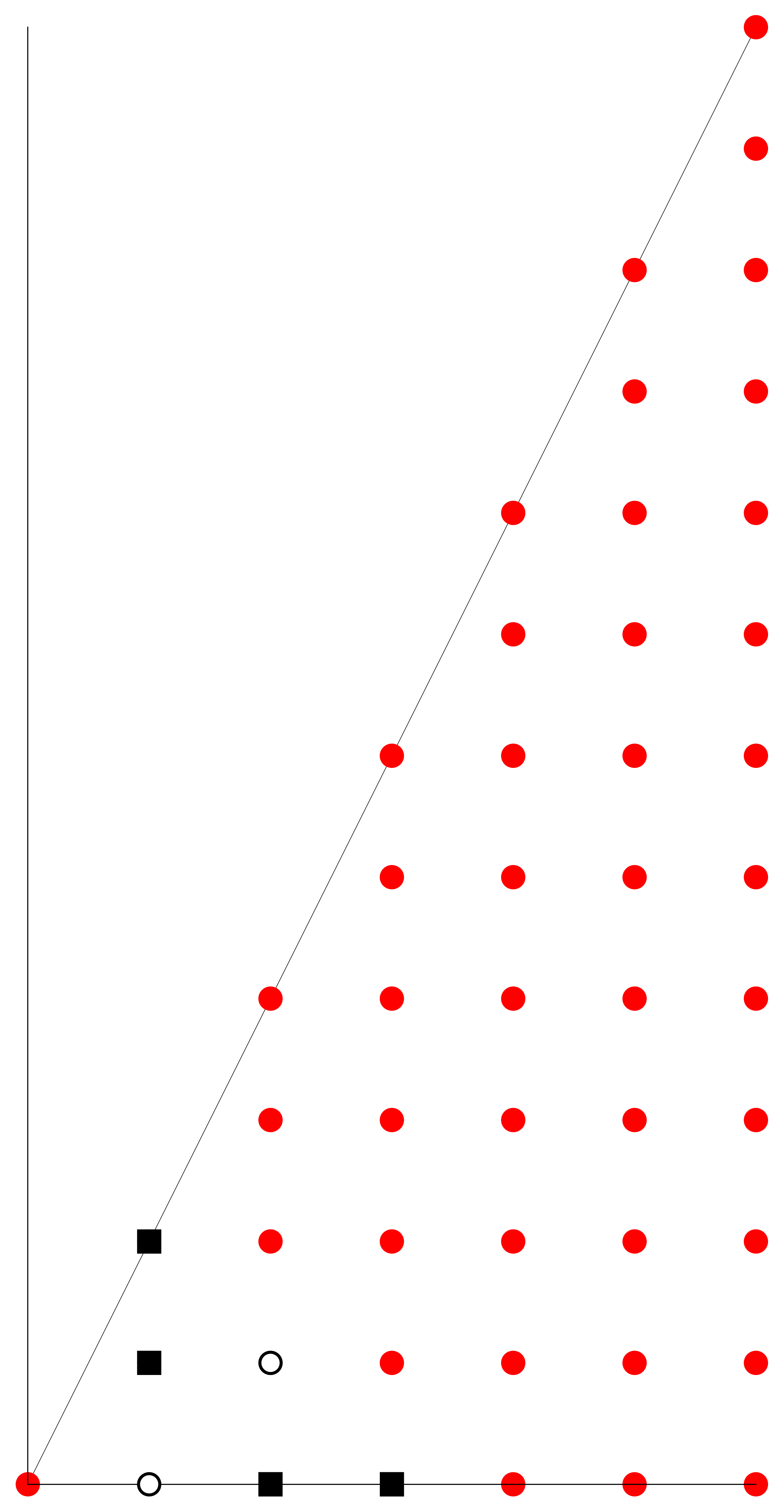} &
    \includegraphics[width=0.3\textwidth]{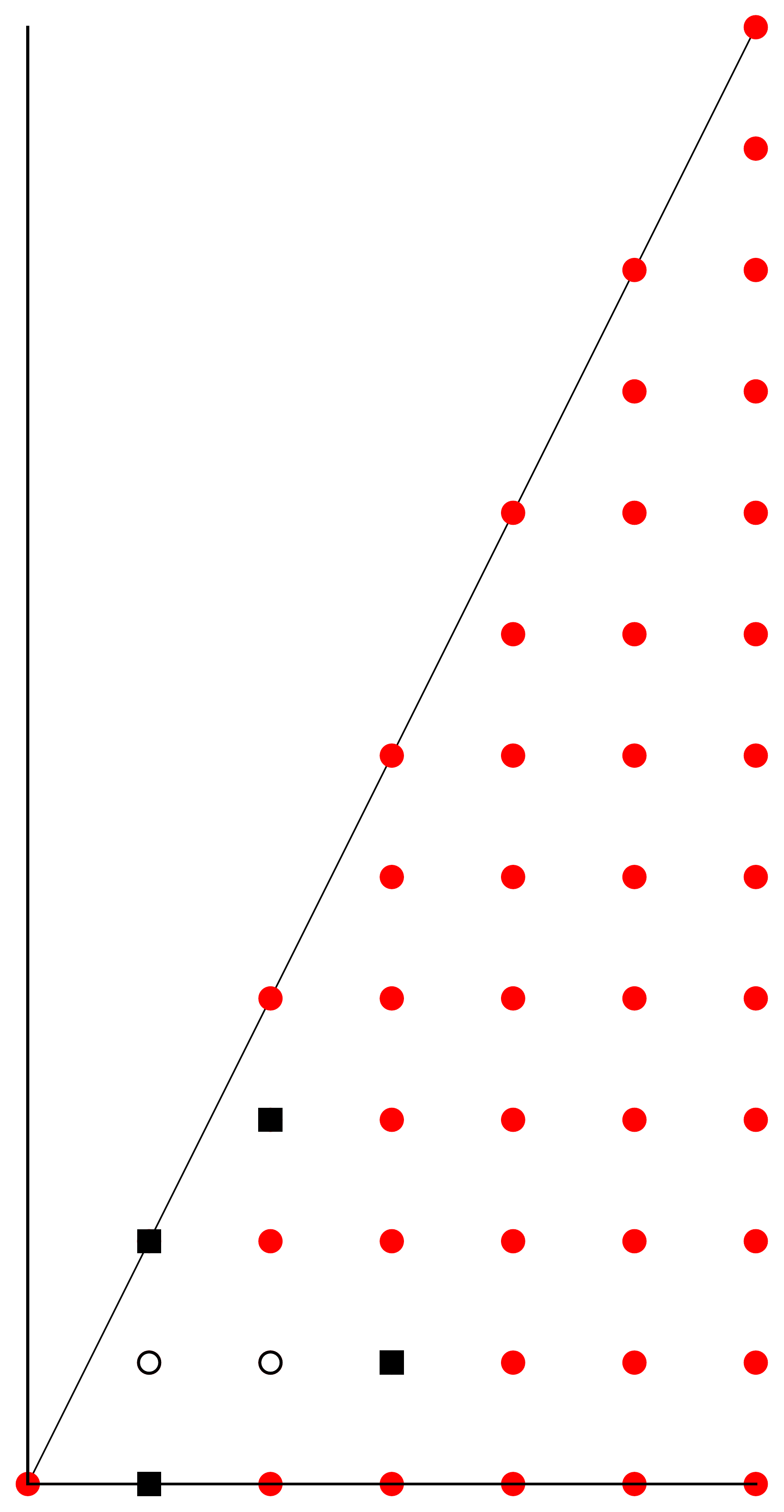} &
    \includegraphics[width=0.3\textwidth]{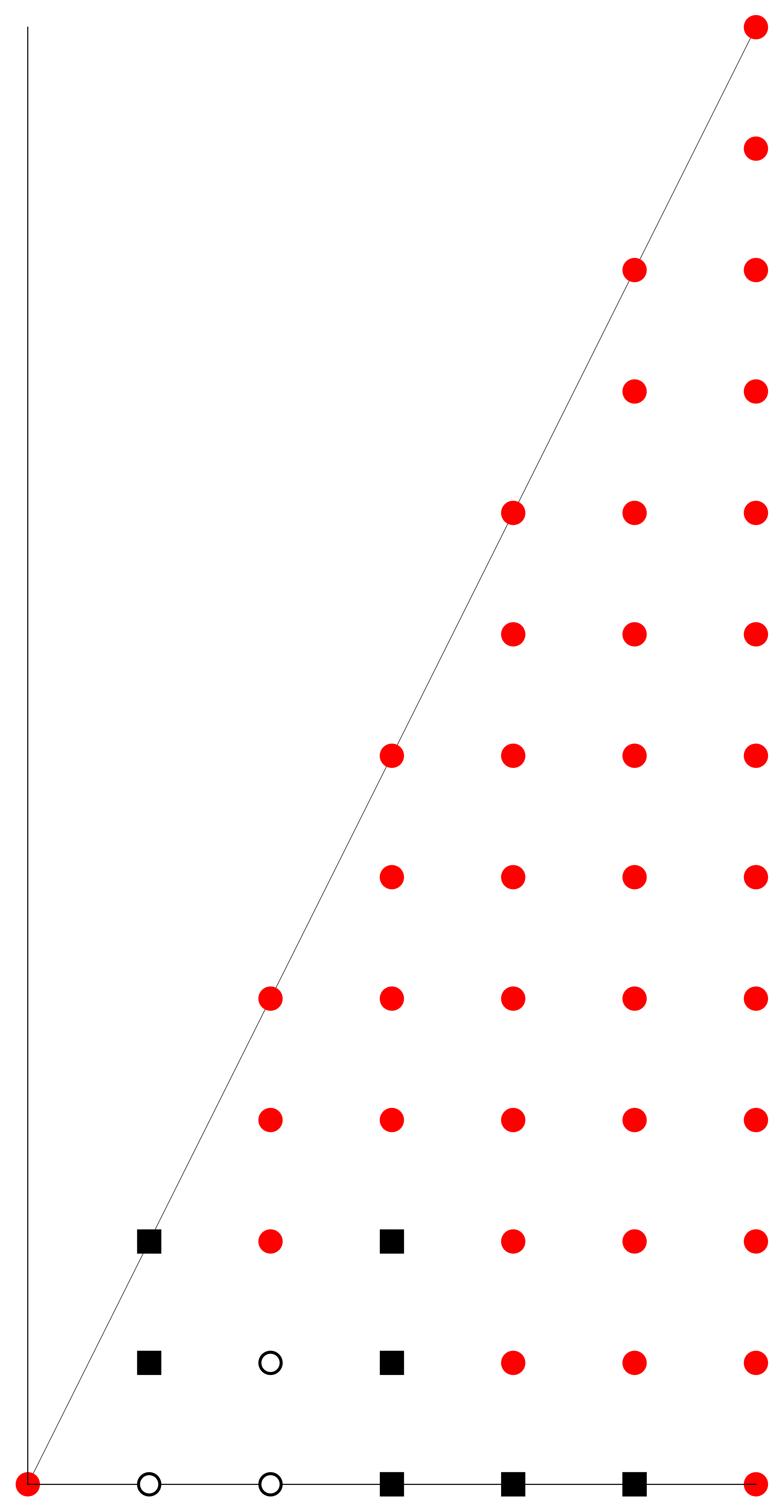}
\\ \hline
        \includegraphics[width=0.3\textwidth]{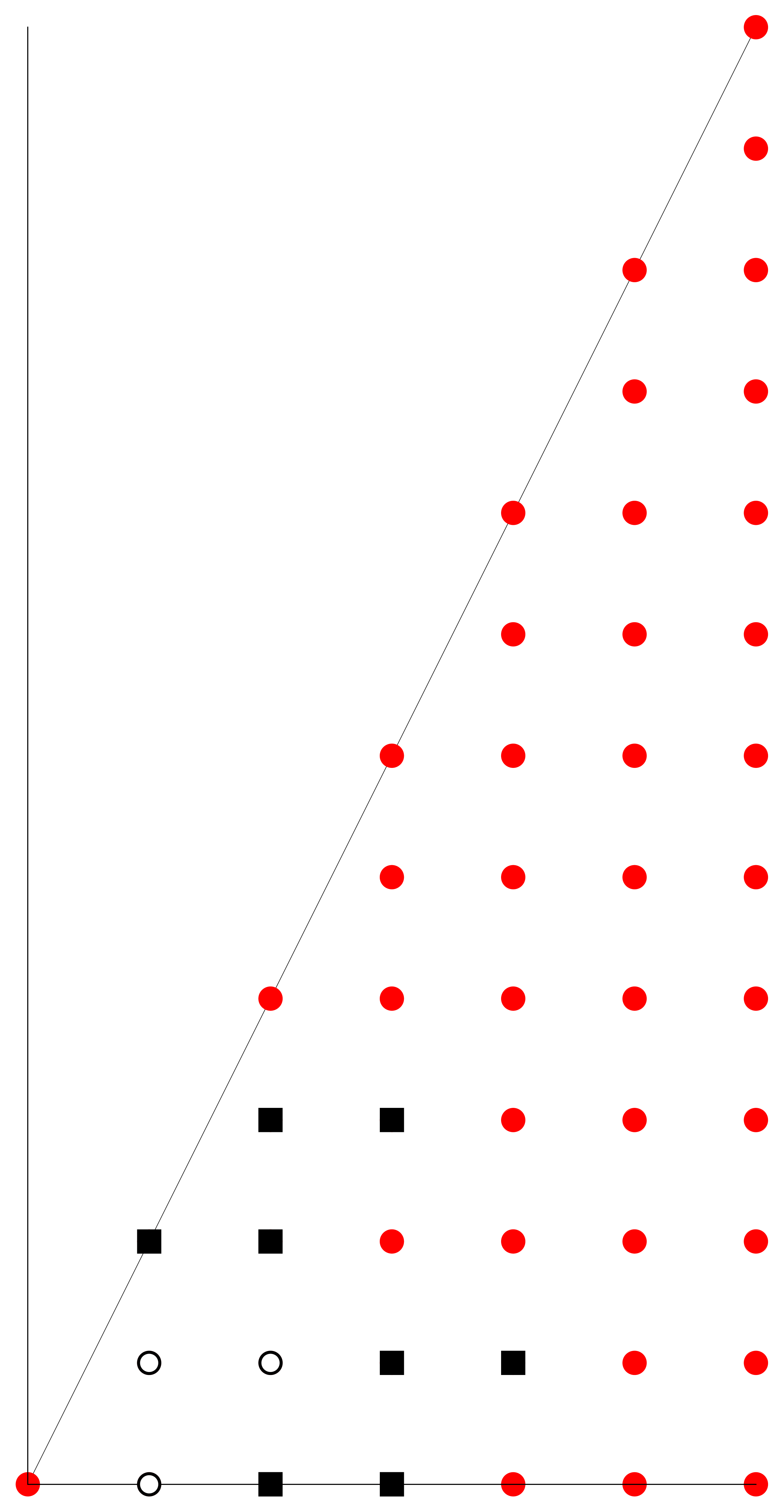} &
    \includegraphics[width=0.3\textwidth]{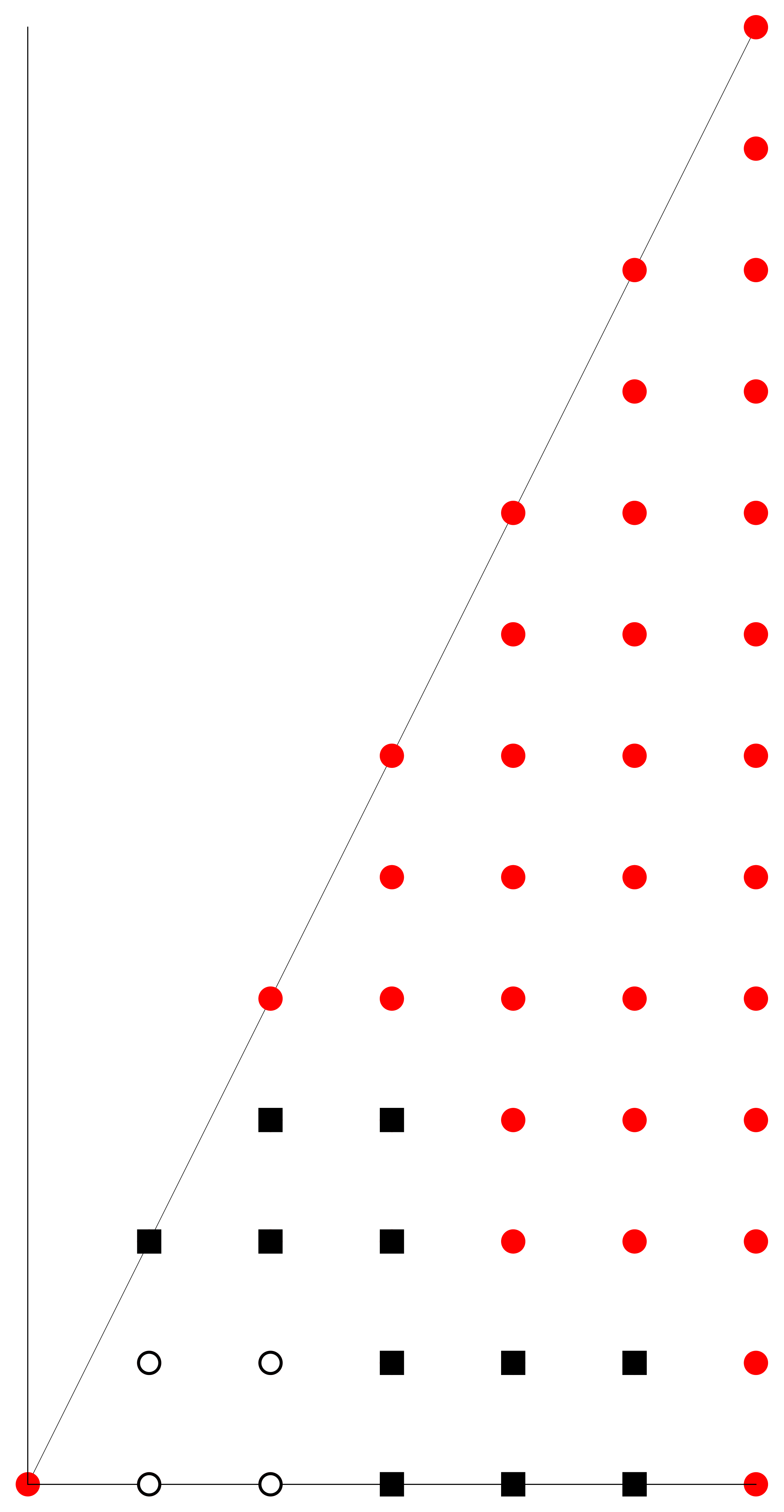} &
    \includegraphics[width=0.3\textwidth]{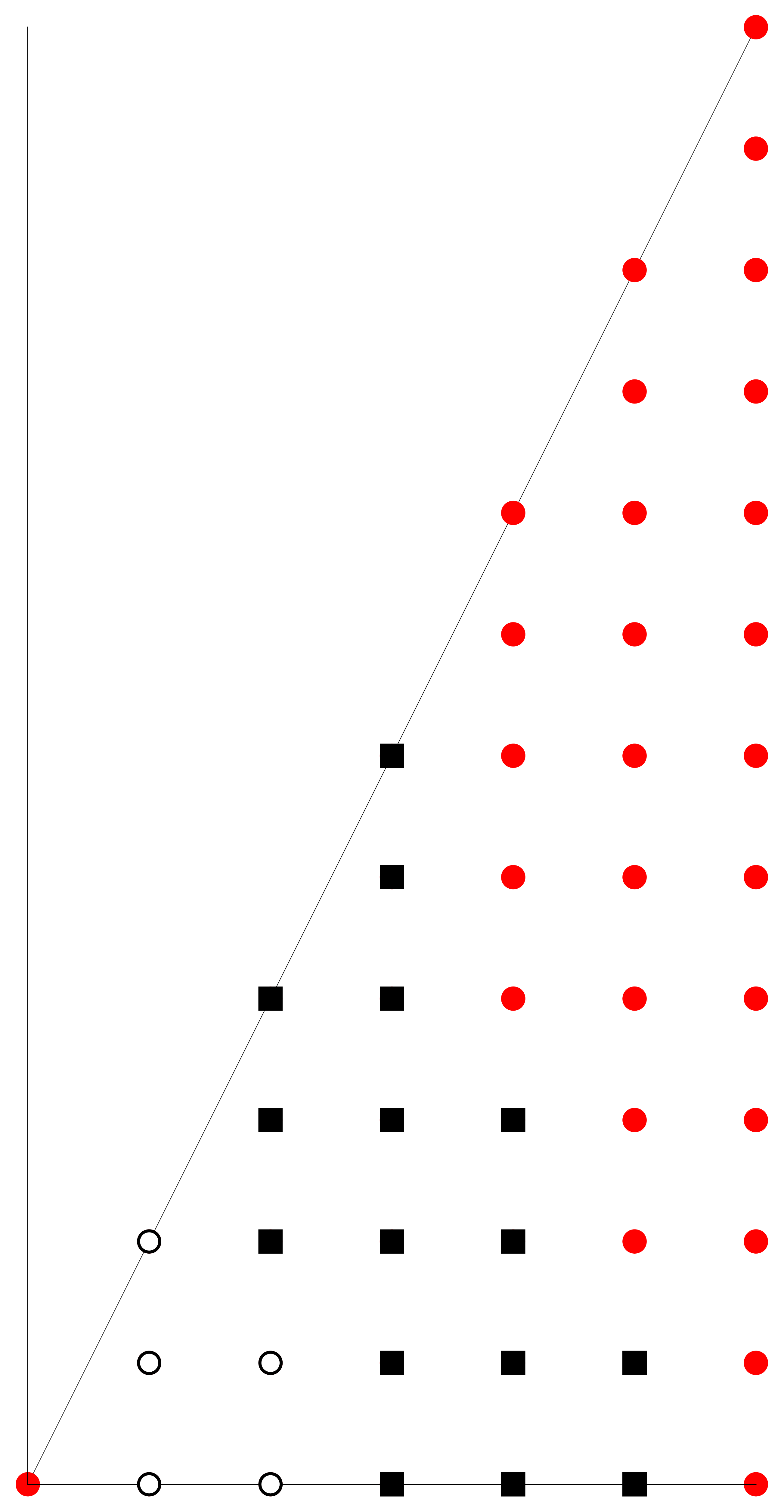}
\\  \hline
    \end{tabular}
        \caption{All $\CaC$-semigroups with $\CaC=\langle (1,0),(1,1),(1,2) \rangle$, and Frobenius vector equal to  $(2,1)$;
        $\circ \equiv  \text{gap}$, $\blacksquare  \equiv \text{minimal generator}$, ${\color{red} \bullet} \equiv \text{element in }S$.
        }
    \label{tab:LastExample}
\end{table}
\end{example}

\subsubsection*{Funding}
The first, second, and last authors were partially supported by Junta de Andaluc\'{\i}a research group FQM-343, and by Consejería de Universidad, Investigación e Innovación de la Junta de Andalucía project ProyExcel\_00868. Proyecto de investigación del Plan Propio – UCA 2022-2023 (PR2022-004) partially supported the second and last authors. Proyecto de investigación del Plan Propio – UCA 2022-2023 (PR2022-011) also partially supported all the authors. This publication and research have been partially granted by INDESS (Research University Institute for Sustainable Social Development), Universidad de Cádiz, Spain.


\subsubsection*{Author information}

J. I. Garc\'{\i}a-Garc\'{\i}a. Departamento de Matem\'aticas/INDESS (Instituto Universitario para el Desarrollo Social Sostenible),
Universidad de C\'adiz, E-11510 Puerto Real  (C\'{a}diz, Spain).
E-mail: ignacio.garcia@uca.es.

D. Marín-Aragón. Departamento de Matem\'aticas,
Universidad de C\'adiz, E-11510 Puerto Real  (C\'{a}diz, Spain).
E-mail: daniel.marin@uca.es.

A. Sánchez-Loureiro. Departamento de Matem\'aticas,
Universidad de C\'adiz, E-11510 Puerto Real  (C\'{a}diz, Spain).
E-mail: a.sanchezlou@alum.uca.es.

A. Vigneron-Tenorio. Departamento de Matem\'aticas/INDESS (Instituto Universitario para el Desarrollo Social Sostenible), Universidad de C\'adiz, E-11406 Jerez de la Frontera (C\'{a}diz, Spain).
E-mail: alberto.vigneron@uca.es.

\end{document}